\documentclass[12pt]{amsart} 
\usepackage{amssymb} 
\usepackage{colordvi} 
\usepackage{graphicx} 
\usepackage{vmargin} 
\setpapersize{USletter} 
\setmargrb{1in}{1in}{1in}{1in} 
\numberwithin{equation}{section} 
\newtheorem{theorem}{Theorem}[section] 
\newtheorem{proposition}[theorem]{Proposition} 
 
\newtheorem{corollary}[theorem]{Corollary} 
 
\newtheorem{defn}[theorem]{Definition}

\theoremstyle{definition}

\newtheorem{example}[theorem]{Example} 
\newtheorem{remark}[theorem]{Remark}

\newcommand{\R}{\mathbb{R}}
\newcommand{\C}{\mathbb{C}}

\newcommand{\N}{\mathbb{N}}
\newcommand{\Q}{\mathbb{Q}}
\newcommand{\PP}{\mathbb{P}}
\newcommand{\G}{{\rm Gr}}

\title{On the convex hull of a space curve}

\author{Kristian Ranestad}
\address{Kristian Ranestad \\
Matematisk Institutt \\
      Universitetet I Oslo \\
PO Box 1053 \\
Blindern, NO-0316 Oslo, Norway}
\email{ranestad@math.uio.no}
\urladdr{http://www.math.uio.no/\~{}ranestad/}
                          
\author{Bernd Sturmfels} 
\address{Bernd Sturmfels\\ 
Department of Mathematics\\ 
      University of California\\ 
      Berkeley, California 94720,
      USA} 
\email{bernd@math.berkeley.edu} 
\urladdr{http://www.math.berkeley.edu/\~{}bernd} 
\begin{document} 

\begin{abstract}
The boundary of the convex hull of a compact algebraic curve
in real $3$-space defines a real algebraic surface. For 
general curves, that boundary 
surface is reducible, consisting of tritangent planes and a scroll of stationary bisecants. We express the
degree of this surface in terms of the degree, genus and 
singularities of the curve. We present algorithms
for computing their defining polynomials, and we exhibit a 
wide range of
examples.
\end{abstract} 

\maketitle

\section{Introduction}

The convex hull of an algebraic space curve in $\R^3$
is a semi-algebraic convex body. The aim of this
article is to examine the boundary surface of such a convex body
using methods from (computational) algebraic geometry.
We shall illustrate our questions and results by way
of a simple first example. Consider the
trigonometric space curve defined parametrically~by 
\begin{equation}
\label{eq:runningcurve1}
x = {\rm cos}(\theta)\,, \,\,       y = {\rm sin}(2\theta)\, , \,\, z = {\rm cos}(3\theta).  
\end{equation}
This is an algebraic curve of degree $6$
cut out by intersecting two surfaces of degree $2$ and $3$:
\begin{equation}
\label{eq:runningcurve2} x^2-y^2-xz \,\,\, = \,\,\,
z-4 x^3+3 x \,\,\, = \,\,\, 0 .
\end{equation}
Figure~\ref{fig:frank}
shows a picture  (due to Frank Sottile \cite[Fig.~3]{SSS}) of
the convex hull of the~curve.
\begin{center}
\begin{figure}
\vskip -0.7cm
\includegraphics[width=6cm]{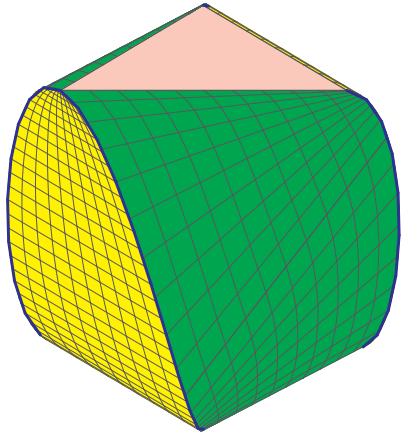}
\vskip -0.4cm
\caption{The convex hull of the curve $({\rm cos}(\theta), {\rm sin}(2\theta),
{\rm cos}(3\theta))$ has
two triangles and two non-linear surfaces patches
of degree $3$ and $16$ in its boundary.}
\label{fig:frank}
\end{figure}
\end{center}

This picture shows the facets (maximal faces) of this convex body.
There are two $2$-dimensional facets, namely the triangles
whose vertices have parameters $\theta = 0,\frac{2\pi}{3}, \frac{4\pi}{3}$ and
$\theta = \frac{\pi}{3}, \pi, \frac{5\pi}{3}$.
Further, we see two one-dimensional families of $1$-dimensional facets.
These sweep out the yellow surface and 
the green surface in Figure \ref{fig:frank}.  
The boundary  is the union of the two triangles and the two surfaces.
Each triangle lies in a tritangent plane to the curve, while the $1$-dimensional facets are segments in stationary bisecant lines.
The union of all stationary bisecant lines is the 
{\em edge surface} of the curve,  as defined in Section 2.  The 
lines in the quadratic cone $\{x^2-y^2-xz=0\}$ are also stationary 
bisecants to $C$, 
so this cone is a third component of the edge surface of $C$,
but it does not contribute to the boundary of its convex hull.
The tritangent planes and the edge surface are 
our primary objects of study.

The problem of computing the convex hull of a space curve
is fundamental for non-linear computational geometry,
but the literature on algorithms is surprisingly sparse. 
One exception is an article on geometric modeling by
Seong {\it et al.} \cite[\S 3]{SEJK} which describes the boundary
in terms of stationary bisecants and tritangents, corresponding to our
surfaces and triangles. 
This description of the boundary surface was also known to 
Sedykh \cite{Sed} who undertook a
detailed study of the singularities arising in the boundary of such a 
$3$-dimensional convex body.

Recent interest in computing convex hulls of algebraic varieties
arose in the theory of semidefinite programming. We refer to the
articles of Gouveia {\it et al.} \cite{GPT}, Henrion \cite{Hen}
and Netzer {\it et al.} \cite{NPS} for details
and references. Their aim is to represent the convex hull 
as the projection of a spectrahedron, or, algebraically,
one seeks to find a lifted LMI representation for the convex hull.
Such a representation always exists when the variety is a rational curve.
This follows from the fact that every non-negative polynomial in one
variable is a sum of squares. Its construction was explained in
\cite[\S 4-5]{Hen} and in \cite[\S 5]{SSS}.

Since trigonometric curves are rational, their convex hull has a lifted LMI representation.
The convex hull of our curve
is the following projection of a $6$-dimensional spectrahedron:
$$
\hbox{{\rm Figure} \ref{fig:frank}} \,\,\, = \,\,\,
\biggl\{ (x,y,z) \in \R^3\,\,| \,\,\exists \,u,v,w \in \R \,:
\,
\begin{pmatrix}
1 & x+ u i & v+y i & z  +  w i\\
x -  u i & 1 & x +  u i & v +  y i \\
v - y i & x -  v i & 1 & x +  u i \\
z - w i & v -  y i & x -   u i & 1 
\end{pmatrix} \,\succeq \, 0\,
\biggr\}.
$$
Here $i = \sqrt{-1}$ and ``$\succeq 0$''
means that this Hermitian $4 {\times} 4$-matrix is positive semidefinite.
A lifted LMI representation solves our problem from the point of view of
convex optimization because it allows us to rapidly
maximize any linear function over the curve. However, that formula is
unsatisfactory from an algebro-geometric perspective because it 
reveals little information about boundary surfaces visible in Figure \ref{fig:frank}.
Our goal is to present practical  tools for computing the 
irreducible polynomials
defining these surfaces, or at least their~degrees.  

It is quite easy to see that the yellow surface in Figure \ref{fig:frank} has degree $3$
and is defined by
\begin{equation}
\label{eq:acubicsurface}
z - 4 x^3 + 3 x  \,\,= \,\, 0. 
\end{equation}
However, it is not obvious
that the green surface has degree $16$ and  its defining polynomial~is

$$
\begin{tiny}
  \begin{matrix}
1024 x^{16}-12032 x^{14} y^2+52240 x^{12} y^4-96960 x^{10} y^6+56160 x^8 y^8+19008 x^6 y^{10}+1296 x^4 y^{12} 
+6144 x^{15} z-14080 x^{13} y^2 z 
\\ -72000 x^{11} y^4 z+149440 x^9 y^6 z+79680 x^7 y^8 z+7488 x^5 y^{10} z+15360 x^{14} z^2 
+36352 x^{12} y^2 z^2+151392 x^{10} y^4 z^2+131264 x^8 y^6 z^2 \\
+18016 x^6 y^8 z^2+20480 x^{13} z^3+73216 x^{11} y^2 z^3+105664 x^9 y^4 z^3+23104 x^7 y^6 z^3+15360 x^{12} z^4+41216 x^{10} y^2 z^4+16656 x^8 y^4 z^4\\
+6144 x^{11} z^5+6400 x^9 y^2 z^5+1024 x^{10} z^6-26048 x^{14}-135688 x^{12} y^2+178752 x^{10} y^4+124736 x^8 y^6-210368 x^6 y^8+792 x^4 y^{10} \\
+5184 x^2 y^{12}+432 y^{14}-77888 x^{13} z+292400 x^{11} y^2 z+10688 x^9 y^4 z-492608 x^7 y^6 z-67680 x^5 y^8 z+21456 x^3 y^{10} z+2592 x y^{12} z \\
-81600 x^{12} z^2-65912 x^{10} y^2 z^2-464256 x^8 y^4 z^2-192832 x^6 y^6 z^2+31488 x^4 y^8 z^2+6552 x^2 y^{10} z^2-40768 x^{11} z^3-194400 x^9 y^2 z^3 \\
-196224 x^7 y^4 z^3+14912 x^5 y^6 z^3+8992 x^3 y^8 z^3-20800 x^{10} z^4-84088 x^8 y^2 z^4-7360 x^6 y^4 z^4+7168 x^4 y^6 z^4-12480 x^9 z^5 \\
-9680 x^7 y^2 z^5+3264 x^5 y^4 z^5-2624 x^8 z^6+760 x^6 y^2 z^6+64 x^7 z^7+189649 x^{12}+104700 x^{10} y^2-568266 x^8 y^4+268820 x^6 y^6 \\
+118497 x^4 y^8-42984 x^2 y^{10}-432 y^{12}+62344 x^{11} z-592996 x^9 y^2 z+421980 x^7 y^4 z+377780 x^5 y^6 z-79748 x^3 y^8 z-18288 x y^{10} z \\
+104620 x^{10} z^2+56876 x^8 y^2 z^2+480890 x^6 y^4 z^2-12440 x^4 y^6 z^2-51354 x^2 y^8 z^2-936 y^{10} z^2+35096 x^9 z^3+181132 x^7 y^2 z^3 \\
+73800 x^5 y^4 z^3-52792 x^3 y^6 z^3-3780 x y^8 z^3-6730 x^8 z^4+52596 x^6 y^2 z^4-19062 x^4 y^4 z^4-5884 x^2 y^6 z^4+y^8 z^4+6008 x^7 z^5 \\
+2516 x^5 y^2 z^5-4324 x^3 y^4 z^5+4 x y^6 z^5+2380 x^6 z^6-1436 x^4 y^2 z^6+6 x^2 y^4 z^6-152 x^5 z^7+4 x^3 y^2 z^7+x^4 z^8-305250 x^{10} \\
+313020 x^8 y^2+174078 x^6 y^4-291720 x^4 y^6+74880 x^2 y^8+84400 x^9 z+278676 x^7 y^2 z-420468 x^5 y^4 z+20576 x^3 y^6 z+40704 x y^8 z \\
-25880 x^8 z^2-76516 x^6 y^2 z^2-148254 x^4 y^4 z^2+77840 x^2 y^6 z^2+5248 y^8 z^2-29808 x^7 z^3-49388 x^5 y^2 z^3+23080 x^3 y^4 z^3 \\
+14560 x y^6 z^3+14420 x^6 z^4-7852 x^4 y^2 z^4+9954 x^2 y^4 z^4+568 y^6 z^4+848 x^5 z^5+92 x^3 y^2 z^5+1164 x y^4 z^5-984 x^4 z^6+724 x^2 y^2 z^6 \\
-2 y^4 z^6+112 x^3 z^7-4 x y^2 z^7-2 x^2 z^8+140625 x^8-270000 x^6 y^2+172800 x^4 y^4-36864 x^2 y^6-75000 x^7 z+36000 x^5 y^2 z \\ +46080 x^3 y^4 z-24576 x y^6 z-12500 x^6 z^2+49200 x^4 y^2 z^2-19968 x^2 y^4 z^2-4096 y^6 z^2+15000 x^5 z^3-10560 x^3 y^2 z^3 \\
-3072 x y^4 z^3  -2250 x^4 z^4-1872 x^2 y^2 z^4
+768 y^4 z^4-520 x^3 z^5+672 x y^2 z^5+204 x^2 z^6-48 y^2 z^6-24 x z^7+z^8.
\end{matrix}
\end{tiny}
$$

This paper is organized as follows.
In Section 2 we apply known results from  algebraic geometry
to derive formulas for the number of tritangents
and the degree of the edge surface of a smooth curve $C$
of degree $d$ and genus $g$ in $\R^3$. This characterizes the intrinsic
algebraic complexity of computing the
convex hull of $C$. In Section 3 we focus on trigonometric curves,
which are compact of even degree $d$ and genus $g=0$.
We describe an algebraic
elimination method for computing their tritangents and
edge surfaces.
The method will be demonstrated for curves of degree $d=6$,
which have $8$ tritangents and whose edge surface has degree $30$. 
If the curve is singular then that number drops, for 
instance to $16{+} 3 {+} 2$
in the example~above. 

Freedman \cite{Fre}
asked in 1980 whether every generic smooth knotted curve $C$
in $\R^3$ must have a tritangent plane.
Ballesteros-Fuster \cite{BF} and Morton \cite{Mor}
answered this to the negative by constructing
trigonometric curves without tritangents.
We reexamine the Morton curve in Section 4. 
Section 5 offers an in-depth study of
the edge surface from the algebraic geometry
perspective. We establish a refined degree formula
that also works for curves with singularities, and we derive both old and
new results on edge surfaces and their dual varieties.

The {\em second hull} of a knotted space curve,
studied in \cite{CKKS}, is
strictly contained in the convex hull, 
but the algebraic surface defining their boundaries coincide.
Thus, our algebraic recipes not only compute and represent the ordinary
convex hull but they also yield the second hull.

Our study of the convex hull has been extended to higher-dimensional 
varieties in \cite{RS}, which is a sequel to the present article.
The paper \cite{RS} contains a characterization of the 
boundary of the convex hull of a compact real variety in affine space 
in terms of certain multiple strata in the dual variety.  When the 
variety is a space curve, the important special case studied here, the dual variety of the edge surface is a 
double curve on the dual surface of the space curve, while the 
tritangent planes correspond to triple points on the dual surface.

\section{Degree Formula for Smooth Curves}

Let $C$ be a compact smooth real algebraic curve in $\R^3$.
This means that the Zariski closure of $C$ in complex
projective space $\C\PP^3$ is a smooth projective curve, denoted $\bar C$.
We define the {\em degree} and {\em genus} of $C$ to be
the corresponding quantities for the complex curve $\bar C \subset \C\PP^3$:
$$  
d = {\rm degree}(C) := {\rm degree}(\bar C)
\quad \hbox{and} \quad
g = {\rm genus}(C) := {\rm genus}(\bar C).$$
Our object of study is the  convex hull
${\rm conv}(C)$ of the real algebraic curve $C$.
This is a compact, convex, semi-algebraic subset of $\R^3$,
and its boundary $\partial {\rm conv}(C)$ is a 
semi-algebraic subset of pure dimension $2$ in $\R^3$.
We wish to understand the structure of this boundary.

We define the {\em algebraic boundary} of the convex body ${\rm conv}(C)$
to be the $K$-Zariski closure of $\partial {\rm conv}(C)$ in complex
affine space $\C^3$. Here $K$ is the subfield of $\R$ over which the
curve $C$ is defined. The algebraic boundary
is denoted  $\partial_a {\rm conv}(C)$. This complex
surface is usually reducible, and we
identify it with its defining square-free polynomial in $K[x,y,z]$.
Note that the algebraic boundary depends on the choice of the field $K$,
and its degree is understood in the usual sense of algebraic geometry.
All curves $C$ in our examples are defined over $K= \Q$.

Combining the description of the convex hull by Sedykh \cite{Sed} and Seung 
{\it et al.}~\cite{SEJK} with 
enumerative results of De Jonqui\`eres \cite{ACGH}, Arrondo {\it et al.}~\cite{ABT}  
and Johnsen \cite {J}, we shall derive the  following characterization of the 
expected factors of this polynomial and their degrees:

\begin{theorem} \label{thm:degreesmooth}
Let $C$ be a general smooth compact curve  
of degree $d$ and genus $g$ in $\R^3$.
The algebraic boundary $\partial_a {\rm conv}(C)$ 
of its convex hull is the union of the
edge surface and the tritangent planes.
The edge surface is irreducible of degree $\,2(d-3)(d+g-1)$,
and the number of complex tritangent planes equals
$\,8\binom{d+g-1}{3}-8(d{+}g{-}4)(d{+}2g{-}2)+8g-8$.
\end{theorem}

We first explain some of the terms appearing in the statement,
and then we embark on the proof.
A plane $H$ in $\C\PP^3$ is a {\em tritangent plane} of $\bar C$ if
$H$ is tangent to $\bar C$ at three or more non-collinear points. For general
curves, no plane $H$ can be tangent to four or more points on $\bar C$,
so all tritangents touch the curve $\bar C$ in precisely three points.
In the count above we assume that this is the case.  
In particular, the curve is not a plane curve.
Among the tritangent planes are the affine spans of two-dimensional
facets of ${\rm conv}(C)$, and generically such facets are triangles.
Usually, not all tritangent planes will be defined over
the real numbers, and only a subset of the real tritangent planes
will correspond to  triangle faces of $\partial {\rm conv}(C)$. Moreover, if $C$ is 
defined over $\Q$, we can use symbolic computation to
compute the polynomial that defines the union of all tritangent planes.
This is the {\em Chow form} in (\ref{eq:chowform}) below.
The number of tritangent planes in
Theorem \ref{thm:degreesmooth} is the degree of that Chow form.

We define the {\em edge surface} of $C$ to be the union of all
stationary bisecant lines in the sense 
of Arrondo {\it et al.} \cite[\S 2]{ABT}.
To see what this means, let us consider any point $p$
in the boundary $\partial {\rm conv}(C)$ that does not lie in 
$C$ or in any $2$-dimensional face. Since
every maximal face of a convex body is exposed, the boundary
$\partial {\rm conv}(C)$ is the union of the exposed faces.
We can thus choose a plane $H$ that exposes a face $F$ containing $p$.
The face $F$ is not a polygon  and it is not a vertex since
$p \not\in C$. Therefore $F$ is one-dimensional, $H \cap C$
consists of two points $p_1 $ and $p_2$, and $F$ is the edge
between $p_1$ and $p_2$. The line $L$ spanned by 
$F$ is a {\em stationary bisecant line}. This means that
$L$ is a bisecant line and that the tangent lines
of $C$ at the intersection points $p_1$ and $p_2$ lie in a common plane, namely $H$.
The edge surface may have several components, as we shall
see in Example \ref{ex:ellipspace}.
As for tritangent planes, only a subset of the stationary bisecant 
lines correspond to $1$-dimensional facets of $\partial {\rm conv}(C)$, 
so the edge surface may have components that do not contribute to this boundary. 

Figure~\ref{fig:quartic} shows 
a smooth rational quartic curve $C$ and
its edge surface. Here
$d=4, g=0$, so Theorem \ref{thm:degreesmooth} says that
there are no tritangents and the edge surface has degree six.
The surface is singular along the curve $C$,
and ${\rm conv}(C)$ is visible in the center of the diagram.

\begin{proof}[Proof of Theorem \ref{thm:degreesmooth}]
The number of tritangent planes will be derived
from {\em De Jonqui\`eres' formula} \cite[p.~359]{ACGH} for a smooth complex 
projective curve $\bar C$ of degree $d$ and genus $g$ in $\C\PP^r$.  
Let $a=(a_{1},...,a_{k})$ and $n=(n_{1},...,n_{k})$ be 
vectors of positive integers with $\sum_{i=1}^ka_{i}n_{i}=d$. 
We assume that the $a_{i}$ are distinct and 
that $\,s :=d-\sum_{i=1}^kn_{i} \leq r$.
The set of all  hyperplanes that intersect $C$ in $d-s$ points,
where $n_{i}$ are intersected with  multiplicity $a_{i}$,
is a variety $V_{a,n}$ in the dual space $(\C\PP^r)^*$.
De Jonqui\`eres' formula states the following:
If the dimension of $V_{a,n}$ is $r-s$ then the degree of  $V_{a,n}$ equals
the coefficient of $t_{1}^{n_{1}}\cdots t_{k}^{n_{k}}$ in the polynomial 
\begin{equation}
\label{eq:dejonquieres}
(1+\sum_{i=1}^ka_{i}^2t_{i})^g \cdot (1+\sum_{i=1}^ka_{i}t_{i})^{d-s-g}.
\end{equation}
This formula can be used to investigate the sets of planes 
that are tangent to a space curve $\bar C \subset \C \PP^3$, so first we set $r=3$. 
The variety of planes tangent to $\bar C$ is 
the {\em dual variety} $\bar C^*$ of $C$ and is obtained 
by taking $a=(2,1), n=(1,d-2)$ and $s=1$. In this
way we recover the result that the
dual variety $\bar C^*$ is a surface of degree $2(d+g-1)$. 

The variety of tritangent planes for a  curve
$\bar C \subset \C \PP^3$ is obtained from (\ref{eq:dejonquieres})
by setting $a=(2,1), n=(3,d-6)$ and $s=3$.
The expected dimension of that variety is  $r-s=0$.
So, when the set of tritangent planes is finite then its
cardinality is the coefficient of $t_{1}^3t_{2}^{d-6}$ in 
$$(1+4t_{1}+t_{2})^g \cdot (1+2t_{1}+t_{2})^{d-3-g}.$$
That coefficient is found to be the desired quantity
$\,8\binom{d+g-1}{3}-8(d{+}g{-}4)(d{+}2g{-}2)+8g-8$.

\smallskip

We argued above that the union of all 
$1$-dimensional facets of ${\rm conv}(C)$ is Zariski dense in
a component of the surface of stationary bisecants. This surface is
precisely the edge surface of $C$, as defined above.
 Arrondo {\it et al.} \cite[\S 2]{ABT} study the edge 
surface as the {\em focal surface} of the congruence of bisecant lines or 
secants  to $C$. They derive the degree of this surface
from \cite[Propositions 1.7 and 2.1]{ABT}. 
The desired formula $2(d-3)(d+g-1)$ is stated explicitly in 
the remark prior to Example 2.4 in \cite[page 547]{ABT}.
See also \cite[Remarks 5.1 and 5.2]{J}. 
We also give a  direct derivation in Proposition \ref{degree}.
This concludes the proof of Theorem \ref{thm:degreesmooth}.
\end{proof}

\begin{figure}
\includegraphics[width=8.7cm]{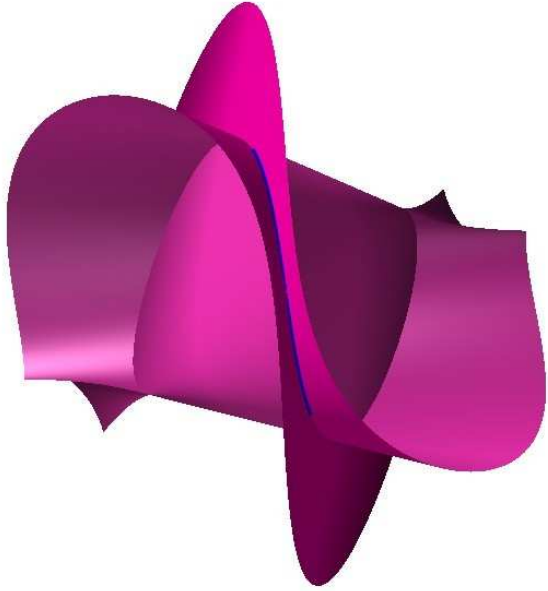}
\vskip -0.7cm
\caption{The edge surface of a rational quartic curve is
irreducible of degree six.}
\label{fig:quartic}
\end{figure}

\begin {remark}\label{stall}
A general space curve $\bar C$ has a finite number of {\em stalls}.
These are planes of third order contact at a point of the curve.
They were studied by Banchoff, Gaffney and McCrory \cite{BGM}.
De Jonqui\`eres' formula with $a=(4,1)$, $n=(1,d-4)$ and $s=3$ 
gives  $4(d+3g-3)$ for the number of stalls.
Related work is Sedykh's classification \cite{Sed}
of six types of singularities.
Each singularity type is exhibited by one of
the two curves in Figures \ref{fig:frank}
and \ref{fig:morton}. \qed
\end{remark}

As noted above, the formulas of Theorem \ref{thm:degreesmooth} do not apply
to plane curves.
For space curves of degree $d \leq 5$, the formulas predict that 
the  number of tritangent planes is zero.

\begin{example}[$d=4, g=1$]  \label{ex:ellipspace}
Consider a compact intersection $\,C = Q_1 \cap Q_2\,$
of two general quadratic surfaces in $\R^3$.
For instance, we could take $Q_1$ and $Q_2$  to be ellipsoids.
The intersection curve $C$ is an {\em elliptic space curve}:
it has genus $g=1$ and degree $d=4$.
According to the formula in Theorem \ref{thm:degreesmooth},
the edge surface of $C$ has degree $8$. 
That surface is not irreducible but is the union of four
quadratic cones. Indeed, the pencil of quadrics $ \,Q_1 + t Q_2\,$
contains precisely four singular quadrics, corresponding
to the four real roots $t_1,t_2,t_3,t_4$ of 
$f(t) = {\rm det}(Q_1 + t Q_2)$.
The rulings of these cones are all stationary bisecants to $C$. Their
union is a surface of degree $8$, and this is
the edge surface of our  elliptic curve $C$.

In algebraic contexts, when $Q_1$ and $Q_2$ have coefficients
in $\Q$, we use symbolic computation to determine
 the algebraic boundary $\,\partial_a {\rm conv}(C)$.
Its defining polynomial is the resultant
$$
\prod_{i=1}^4(Q_1 + t_i Q_2)(x,y,z) 
\,\,= \,\,{\rm resultant}_t \bigl(\, f(t),\, (Q_1+t Q_2)(x,y,z) \,\bigr).
$$
We note that each of the four singular quadrics is the determinant
of a linear symmetric $2 {\times} 2$-matrix  polynomial
$A_i x + B_i y + C_i z + D_i$. Placing these 
matrices along the diagonal in an $8 {\times} 8$ matrix of four
$2 {\times 2}$-blocks, we
obtain a representation of ${\rm conv}(C)$ as 
spectrahedron. \qed
\end{example}

This example shows that the edge surface of a curve
$C \subset \R^3$  can have multiple components even if its
complexification $\bar C \subset \C \PP^3$ is
smooth and irreducible.
We conjecture that at most one of these components is not
a cone. For more information see
 Proposition \ref{comp} below.

\section{Trigonometric Curves and their Edge Surfaces}\label{trig}

By a  {\em trigonometric polynomial} of degree $d$ we mean an expression of the form
\begin{equation}
\label{eq:trigpoly}
f(\theta) \,\,\,=\,\, \,
\sum_{j=1}^{d/2} \alpha_j \,{\rm cos}(j \theta) \,\,+\,\,
\sum_{j=1}^{d/2} \beta_j \,{\rm sin}(j \theta) \,\, + \,\, \gamma.
\end{equation}
Here $d \in \N$ is tacitly assumed to be even and 
the coefficients $\,\alpha_j, \beta_j, \gamma\,$
can be arbitrary real numbers.
We regard $f(\theta)$ as a real-valued function on the unit circle.
A {\em trigonometric space curve} of degree $d$ is
a curve parametrized by three
trigonometric polynomials of degree~$d$:
\begin{equation}
\label{eq:trigcurve}
C \,\, = \,\, \bigl\{\,
\bigl( f_1(\theta), f_2(\theta), f_3(\theta) \bigr) \in \R^3 \,: \,\theta \in [0,2\pi] \,\bigr\} .
\end{equation}
The curve $C$ is the image of the circle under a polynomial map, so it is
clearly compact.

For general coefficients $\alpha_j,\beta_j,\gamma$, the corresponding
complex projective curve $\,\bar C \subset \C \PP^3\,$ is smooth of degree $d$ and it has
genus $g = 0$. As in \cite[\S 5]{Hen},
we can derive a polynomial parametrization of the algebraic curve $\bar C$ 
by means of the following  change of coordinates:
\begin{equation}
\label{eq:coordchange}
\cos(\theta) = \frac{x_{0}^2-x_{1}^2}{x_{0}^2+x_{1}^2}
\quad \hbox{and} \quad \sin(\theta)= \frac{2x_{0}x_{1}}{x_{0}^2+x_{1}^2}. 
\end{equation}
Substituting into the right hand side of the equation
$$ 
\begin{pmatrix}
\phantom{-}{\rm cos}(j \theta)  & {\rm sin}(j \theta) \, \\
- {\rm sin}(j \theta)  & {\rm cos}(j \theta) \,
\end{pmatrix} \,\, = \,\,
\begin{pmatrix}
\phantom{-}{\rm cos}( \theta)  & {\rm sin}( \theta) \, \\
- {\rm sin}( \theta)  & {\rm cos}( \theta) \,
\end{pmatrix}^j,
$$
 this change of variables expresses ${\rm cos}(j \theta)$ and ${\rm sin}(j \theta)$
as homogeneous rational functions of degree $0$ in $(x_0:x_1)$.
Their common denominator equals  $\,g(x_0,x_1) = (x_0^2 + x_1^2)^d$.
This gives
$$\bar C\,=\,
\bigl\{ (F_{0}(x):F_{1}(x):F_{2}(x):F_{3}(x))
\,= \,(g:gf_{1}:gf_{2}:gf_{3})\in \C\PP^3 \,\,:\,\,(x_{0}:x_{1})\in \C\PP^1\,\bigr\}.$$ 

In this section, we are interested in computing the  
compact convex body ${\rm conv}(C)$.
The curve $\bar C$ is rational, 
and it is smooth for general choices of $\alpha_j, \beta_j, \gamma \in \R$.
Theorem \ref{thm:degreesmooth} implies:

\begin{corollary} \label{cor:trigcurve}
The algebraic boundary of the convex hull  of a general
rational curve of degree $d $ consists of
$8 \binom{d-3}{3} $ tritangent planes and the
edge surface of degree  $2(d-3)(d-1)$.
\end{corollary}

In what follows we shall explain a symbolic elimination method for computing the
defining polynomial of the edge surface of a rational curve $\bar C$. Our
examples were computed with {\tt Macaulay2} \cite{M2}.
The problem of finding
the tritangent planes is addressed in the next section.

Our approach will involve the Grassmannian ${\rm Gr}(2,4)$, which
parametrizes all lines in $\C \PP^3$. 
We identify ${\rm Gr}(2,4)$ with the hypersurface
in $\C \PP^5$ that is cut out by the Pl\"ucker quadric
$$ u_{01} u_{23} - u_{02} u_{13} + u_{03} u_{12} \, = \, 0 . $$
Consider any pair of points $p$ and $q$ on the curve $\bar C$.
They are represented by points $x_p = (x_{p0}:x_{p1})$
and $x_q = (x_{q0}:x_{q1})$ in the parameter space $\C \PP^1$. The secant line of $C$ through
$p$ and $q$
is the line in $\C \PP^3$ whose six Pl\"ucker coordinates $\,u_{ij} \,$
are the $2 {\times}2 $-minors of the matrix
$$\begin{pmatrix}
\,F_{0}(x_{p})&F_{1}(x_{p})&F_{2}(x_{p})&F_{3}(x_{p}) \, \\
\,F_{0}(x_{q})&F_{1}(x_{q})&F_{2}(x_{q})&F_{3}(x_{q}) \, \\
\end{pmatrix}.
$$

The six minors are polynomials of degree $2d$ in the coordinates $x_{pj}$ and $x_{qj}$,
and they share the common factor $x_{p0}x_{q1}-x_{p1}x_{q0}$. 
Dividing each minor by this factor yields polynomials $u_{ij}$ 
of degree $2d-2$.
They are bihomogeneous of degree $(d-1,d-1)$
in the $x_{pj}$ and $x_{qj}$, and they are
invariant under permuting the points $p$ and $q$. We can write
each polynomial $u_{ij}$ uniquely
as a polynomial of degree $d-1$ in the three
fundamental bihomogeneous invariants 
\begin{equation}
\label{eq:invariants}
a\,=\,x_{p0}x_{q0}\,, \,\,
b\,=\, x_{p1}x_{q1}\,, \,\,\,
c\,=\,x_{p0}x_{q1}+x_{p1}x_{q0}. 
\end{equation}
This identifies the plane
$\C \PP^2$ with coordinates $(a:b:c)$ 
with the symmetric square of the parameter line
$\C \PP^1$ of our curve $\bar C$.
We have constructed a polynomial map of  degree $d-1$:
\begin{equation}
\label{eq:secantmap}
\C \PP^2 \,\rightarrow \, {\rm Gr}(2,4) \subset \C \PP^5 \,\,, \,\, \,
(a:b:c) \,\mapsto\, ( u_{01}:u_{02}: u_{03}:u_{12}:u_{13}:u_{23}). 
\end{equation}
Geometrically, this represents
the secant map
from the symmetric square of the space curve $\bar C$ to the Grassmannian.
The image of this map is a surface of degree $(d-1)^2$ in ${\rm Gr}(2,4)$.

The stationary bisecants to $\bar C$ are the secant lines between points 
$p,q$ such that the tangent lines to $\bar C$ at these points intersect.
The tangent line at $p$ is defined by the partial derivatives
$$\begin{pmatrix}
\frac{\partial}{\partial x_{p0}}F_{0}(x_{p})&\frac{\partial}{\partial x_{p0}}F_{1}(x_{p})&
\frac{\partial}{\partial x_{p0}}F_{2}(x_{p})&\frac{\partial}{\partial x_{p0}}F_{3}(x_{p})\\
\rule{0pt}{15pt}
\frac{\partial}{\partial x_{p1}}F_{0}(x_{p})&\frac{\partial}{\partial x_{p1}}F_{1}(x_{p})&
\frac{\partial}{\partial x_{p1}}F_{2}(x_{p})&\frac{\partial}{\partial x_{p1}}F_{3}(x_{p})\\
\end{pmatrix}.
$$
The secant line between the points $p$ and $q$ is 
{\em stationary} if the determinant of the matrix
\begin{equation}
\label{eq:4x4det}
\begin{pmatrix}
\frac{\partial}{\partial x_{p0}}F_{0}(x_{p})&\frac{\partial}{\partial x_{p0}}F_{1}(x_{p})&
\frac{\partial}{\partial x_{p0}}F_{2}(x_{p})&\frac{\partial}{\partial x_{p0}}F_{3}(x_{p})\\
\rule{0pt}{15pt}
\frac{\partial}{\partial x_{p1}}F_{0}(x_{p})&\frac{\partial}{\partial x_{p1}}F_{1}(x_{p})&
\frac{\partial}{\partial x_{p1}}F_{2}(x_{p})&\frac{\partial}{\partial x_{p1}}F_{3}(x_{p})\\
\rule{0pt}{15pt}
\frac{\partial}{\partial x_{q0}}F_{0}(x_{q})&\frac{\partial}{\partial x_{q0}}F_{1}(x_{q})&
\frac{\partial}{\partial x_{q0}}F_{2}(x_{q})&\frac{\partial}{\partial x_{q0}}F_{3}(x_{q})\\
\rule{0pt}{15pt}
\frac{\partial}{\partial x_{q1}}F_{0}(x_{q})&\frac{\partial}{\partial x_{q1}}F_{1}(x_{q})&
\frac{\partial}{\partial x_{q1}}F_{2}(x_{q})&\frac{\partial}{\partial x_{q1}}F_{3}(x_{q})\\
\end{pmatrix}
\end{equation}
vanishes.   The extraneous factor  $x_{p0}x_{q1}-x_{p1}x_{q0}$ appears with multiplicity $4$ in the determinant. Removing this factor of degree $8$ from the determinant, we obtain a
symmetric polynomial of degree $4(d-3)$ in $x_p$ and $x_q$. As before,
we now write the resulting expression as a polynomial $\Phi(a,b,c)$ of degree $2(d-3)$
in the three fundamental invariants  (\ref{eq:invariants}).

The equation $\Phi(a,b,c) = 0$ defines a curve of degree $2(d-3)$
in the plane $\C \PP^2$ that parametrizes the symmetric square of $\bar C$. 
This is the curve of all stationary bisecant lines to $\bar C$.
The image of this plane curve under the degree $d-1$ map (\ref{eq:secantmap})
is a curve of degree $2(d-3)(d-1)$ in the Grassmannian ${\rm Gr}(2,4)$.

We can compute the ideal $I_\Phi \subset \R[u_{01},u_{02},u_{03},u_{12},u_{13},u_{23}]$
of this image curve using Gr\"obner-based elimination. Finally, the last step is
to pass from the curve in ${\rm Gr}(2,4)$ to the 
corresponding ruled surface of the same degree in $\C \PP^3$.
We can do this by
adding three more unknowns $x,y,z$ and by
augmenting the ideal $I_\Phi$ with the four equations
\begin{equation}
\label{eq:skewsymm}
\begin{pmatrix}
\phantom{-} 0 & \phantom{-}u_{23} & -u_{13} & \phantom{-} u_{12} \, \\
-u_{23} & \phantom{-} 0 & \phantom{-} u_{03} &- u_{02} \, \\
\phantom{-} u_{13} & - u_{03} & \phantom{-} 0 & \phantom{-} u_{01} \, \\
-u_{12} & \phantom{-} u_{02}  & -u_{01} & \phantom{-} 0 \,
\end{pmatrix} \cdot 
\begin{pmatrix} 1 \\ x \\ y \\ z \end{pmatrix} \,\, = \,\,
\begin{pmatrix} 0 \\ 0 \\ 0 \\ 0 \end{pmatrix}.
\end{equation}
The resulting ideal lives in $ \R[u_{01},u_{02},u_{03},u_{12},u_{13},u_{23},x,y,z]$.
From that ideal, we eliminate the first six unknowns to arrive at a principal ideal
in $\R[x,y,z]$. The polynomial which generates the principal elimination ideal is the
defining equation of the edge surface of $C$.

\begin{example}[$d=4, g=0$]
\label{ex:quarticcurve}
We consider the convex hull of  a general  trigonometric curve $C$ of degree four.
There are no tritangents, so its algebraic boundary is just the
edge surface. In contrast to the elliptic curves of Example \ref{ex:ellipspace},
where the edge surface had four components,
the edge surface of the rational quartic curve $C$ is irreducible of degree six.
For example, let
\begin{equation}
\label{eq:quarticcurve}
C \,:\, \theta \,\mapsto \, \bigl( 
{\rm cos}(\theta), {\rm sin}(\theta) + {\rm cos}(2 \theta), {\rm sin}(2 \theta) \bigr) .
\end{equation}
This curve is smooth and it is cut out by one quadric and two cubics.
Its prime ideal is
$$  \bigl\langle
3 x^2 - y^2 + 2 x z - z^2 - 2 y,\,
   2 y^3-4 x y z+2 y z^2+2 y^2+x z+z^2-4 y,
   4 x y^2+4 x z^2-2 x y-4 y z-2 x+z \bigr\rangle.
$$
The degree $3$ map $\C \PP^2
\rightarrow {\rm Gr}(2,4)$ in (\ref{eq:secantmap})  which parametrizes the
secant lines is given by
\begin{small}
$$
\begin{matrix}
u_{01} = -a^2 c-c^3+2 a bc- b^2 c, \qquad \qquad \,  &
u_{02} = a^3-4 a^2 c+ ac^2 -3 a^2 b+3 a b^2-b^3+4  b^2 c -b c^2 ,\\
u_{03} = 2 a^3{-}2 a c^2 {-}2 a^2 b{-}2 a b^2{+}2 b^3{-} 2 b c^2  , &
u_{12} =  a^3{-}3 a^2 c{+}c^2 a{+}c^3{-}a^2 b{-}a b^2{-}2 c a b{+}b^3{-}3  b^2 c {+} b c^2 ,
\\
u_{13} =  2 a^3{-}2 a c^2 {+}2 a^2 b{-}2 a b^2{-}2 b^3{+} 2 b c^2 , &
\! u_{23} =  2 a^3-2 a c^2 +14 a^2 b+14 a b^2-8  a b c+2 b^3-2 b c^2 .
\end{matrix}
$$
\end{small}
The determinant (\ref{eq:4x4det}) reveals
that the curve of stationary bisecants in $\C \PP^2$ is cut out by 
$$ 
\Phi(a,b,c) \quad = \quad 
a^2
-2ab
+4ac
+b^2
+4bc
-c^2.
$$

The image of this curve 
in $\,{\rm Gr}(2,4) \subset \C \PP^5\,$
is a curve of degree six. Its homogeneous prime
ideal  is minimally generated by eight quadrics 
in $\R[u_{01},u_{02},u_{03},u_{12},u_{13},u_{23}]$.
We next add the four equations in (\ref{eq:skewsymm}) to these 
eight quadrics, we saturate with respect to the irrelevant
ideal $\langle u_{01},u_{02},u_{03},u_{12},u_{13},u_{23} \rangle$,
and thereafter we eliminate the six $u_{ij}$'s. As result we find
\begin{small}
\begin{eqnarray*}
16 x^6-32 x^4 y^2+16 x^2 y^4-96 x^5 z-160 x^3 y^2 z+192 x^4 z^2+
    16 x^2 y^2 z^2-128 x^3 z^3+216 x^4 y +48 x^2 y^3 \\ -8 y^5+72 x^3 y z+88 x y^3 
    z-72 x^2 y z^2-8 y^3 z^2+72 x y z^3-207 x^4-138 x^2 y^2-23 y^4 {+}180 x^3 z {+}
    60 x y^2 z \\ -126 x^2 z^2-54 y^2 z^2+108 x z^3-27 z^4-36 x^2 y+4 y^3-36 x y z
    +108 x^2+36 y^2-108 x z+27 z^2.
\end{eqnarray*}
\end{small}
This  irreducible sextic, 
defining the edge surface
of the curve (\ref{eq:quarticcurve}),
is shown in Figure~\ref{fig:quartic}.
\qed
\end{example}

The next examples we shall examine are trigonometric curves $C$
of degree $d=6$. If a rational sextic curve $\bar C$ is smooth, then its edge surface
is irreducible of degree $30$, and our algorithm above
will generate the irreducible  polynomial of degree $30$ 
in $\R[x,y,z]$. However, when the coefficients of $f_1,f_2,f_3$
in (\ref{eq:trigcurve}) are special, then the 
projective curve $\bar C$ may have singularities in $\C \PP^3$,
even if $C$ is smooth in $\R^3$. In those cases, 
the degree of the edge surface of $C$ drops below $30$,
and the surface may even decompose into several components.

\begin{example}\label{ex:runningex}
In the Introduction we discussed the special curve
$\,\bigl({\rm cos}(\theta) , {\rm sin}(2 \theta), {\rm cos}(3 \theta) \bigr)$.
The corresponding polynomial parametrization
$\,\C \PP^1 \rightarrow  \C \PP^3\,$ of the projective curve $\bar C$ equals
$$  \bigl( \,(x_0^2+x_1^2)^3 : (x_0^2-x_1^2)(x_0^2+x_1^2)^2 :
4(x_0^2-x_1^2)(x_0^2+x_1^2)x_0x_1 :
(x_0^2-x_1^2) (x_1^4-14 x_1^2 x_0^2 + x_0^4) \, \bigr). $$
This curve has two singular points.
The double point $(1:0:0:0)$ lies hidden inside the convex hull in Figure~\ref{fig:frank},
and the double point $(0:0:0:1)$ lies in the plane at infinity.  Notice that the plane at infinity 
intersects the curve only in this singular point.
Our algorithm finds six polynomials $u_{ij}(a,b,c)$ of degree five
that define the secant map $\C \PP^2 \rightarrow {\rm Gr}(2,4)$, but now the
curve of stationary bisecants turns out to be reducible:
$$ \Phi \,\,\, =  \,\,\, (a-b) \,c\, (3a^4-6a^2b^2+2a^2c^2+3b^4+2b^2c^2-c^4). $$
Its image in ${\rm Gr}(2,4)$ is a reducible curve whose three components have
degrees $2$, $3$ and $16$. These translate into three irreducible
components of the edge surface of $\bar C$. From the factor $ c $
we obtain the cubic surface  (\ref{eq:acubicsurface}) which is yellow in
Figure \ref{fig:frank}, and from the quartic factor of $\Phi$ we obtain the
surface of degree $16$ which is green in Figure~\ref{fig:frank}.
Finally,  the factor $a-b $ contributes
the quadratic surface $\{x^2-y^2-xz = 0\}$,  
which is not needed for $\partial {\rm conv}(C)$. \qed
\end{example}

\begin{example} \label{ex:Henrion}
Henrion \cite[\S 5]{Hen} discusses the trefoil knot with parametric representation
$$ C \,:\, 
\theta \,\mapsto \, \bigl(\, {\rm cos}(\theta) + 
2 \,{\rm cos}(2 \theta), \,{\rm sin}(\theta ) + 2 \,{\rm sin}(2 \theta),\, 2 \, {\rm sin}(3 \theta)\, \bigr) . $$
This rational sextic $\,\bar C \subset \C \PP^3\,$ 
has one singular point at $(0:0:0:1)$.
By running our algorithm, we find that 
the edge surface of the curve $C$ is irreducible of degree $27$. \qed
\end{example}

\begin{example}
The following remarkable sextic curve $\bar C$ is due to
Barth and Moore \cite{BaMo}:
$$ \C \PP^1 \rightarrow \C \PP^3 \, , \,\,\,
(x_0:x_1) \mapsto 
\bigl(\, x_0^6-2 x_0 x_1^5 \,:\, 2 x_0^5 x_1+x_1^6 \,: \,
x_0^4 x_1^2 \, : \, x_0^2 x_1^4 \,\bigr). $$
This curve is smooth in $\C \PP^3$, so its edge surface
should have the expected degree $30$. However, when
we run our algorithm, its output is only one irreducible polynomial
of degree $10$:
$$ 27 x^5y^5 + 3125y^{10}-1875 x^2 y^7z  +  \cdots
+27z^5-16y^3z. $$
This surprising output is  not  a contradiction. The edge
surface of the curve $\bar C$ naturally carries a 
non-reduced structure of multiplicity three. 
The geometry of $\bar C$ is such that
every stationary bisecant line is a trisecant line with all three 
tangents lying in the same plane.  The curve $\bar C$ therefore has a 
one-dimensional family of planes that are tangent at three points.  
This family defines an irreducible curve $\Gamma$ of degree $6$ in the
dual projective space $(\C \PP^3)^*$.

The curve $\bar C$ has no tritangent planes at all. Indeed,
our definition of tritangent planes required the presence of non-collinear points
of tangency. De Jonqui\`eres formula, which predicts
$8$ tritangents, does not apply here. If we run the algorithm of Section 4 
for computing all points $(\alpha:\beta:\gamma:\delta) \in (\C \PP^3)^*$
dual to tritangent planes then the output is the curve~$\Gamma$.

We can construct a  real compact model $C \subset \R^3$ of
the Barth-Moore curve $\bar C$
by replacing the first coordinate in $\C \PP^3$
with the sum of all four coordinates.
That sum is a positive binary sextic, and the corresponding new hyperplane
at infinity contains no real points. 
This curve is not trigonometric
because $ (x_0^2+x_1^2)^3$ is not in the linear span of
the four coordinates.
\qed
\end{example}

\section{Computing Tritangent Planes}

According to Corollary \ref{cor:trigcurve}, a general
trigonometric curve $C$ of degree $d $ 
has $8 \binom{d-3}{3} $ tritangent planes, and these account for the
two-dimensional polygonal faces of ${\rm conv}(C)$. In this section, we explain how
these planes can be computed symbolically.
The focus of our exposition remains on
sextic curves, where the expected number  of tritangent planes is eight.

We work with the polynomial parameterization 
$(F_0(x): F_1(x): F_2(x): F_3(x))$ of the projective curve $\bar C$.
Each $F_i(x)$ is a binary form of degree $d$ in $x=(x_0,x_1)$.
We represent a plane in $\C \PP^3$ by a linear equation 
$\,\alpha  + \beta x + \gamma y + \delta z = 0 $. The corresponding point
$(\alpha:\beta:\gamma:\delta)$   in the dual projective space
 $(\C \PP^3)^*$
represents the parameters of our problem.
The plane is tangent to $\bar C$ at a point $p$ if the point
$\,(x_{p0}: x_{p1}) \in \C \PP^1\,$ is a double root of 
the binary form
\begin{equation}
\label{eq:binaryform}
\alpha F_0 (x) \,+ \,\beta F_1 (x) \,+\,\gamma F_2(x) \,+\delta F_3(x) .
\end{equation}
We seek to compute the set $\mathcal{T}_C$ of
all points  $\,(\alpha:\beta:\gamma:\delta) \in 
(\C \PP^3)^*\,$ such that (\ref{eq:binaryform}) has three double roots.
The set  $\mathcal{T}_C$ is finite, of cardinality $8 \binom{d-3}{d}$, and we can compute its
ideal using  Gr\"obner-based elimination. An alternative representation of 
$\,\mathcal{T}_C\,$ is its {\em Chow form}
\begin{equation}
\label{eq:chowform}
 \prod_{(\alpha:\beta:\gamma:\delta) \in \mathcal{T}_C} \!\!\!
(\alpha + \beta x + \gamma y + \delta z) .
\end{equation}
If the $F_i(x)$ have coefficients in $\Q$ then
so does the Chow form (\ref{eq:chowform}), which will typically
be irreducible over $\Q$. The corresponding
surface in $\C \PP^3$ is the union of all tritangent planes,
and these include the planes in $\R^3$ that are spanned by
all $2$-dimensional facets of ${\rm conv}(C)$.

For small values of $d$, such as $d=6$, we use the
following preprocessing step
whose output facilitates computing $\mathcal{T}_C$ for all
rational curves $C$ of degree $d$. Consider the binary form
$$
G(x) \,\,\, = \,\,\, \sum_{i=0}^d \kappa_i \,x_0^i x_1^{d-i} 
$$
whose coefficients $\kappa_i$ are unknowns. 
We precompute the prime ideal $P_d$ of height $3$ in
$\Q[\kappa_0,\ldots,\kappa_d]$ whose variety consists of all
binary forms $G(x)$ that have three double roots.
Suppose we know a list of generators for $P_d$.
For any particular curve $C$, we can then
equate $G(x)$ with (\ref{eq:binaryform}), and this
gives an expression for $\kappa_i$ as  a linear combination
of $\alpha, \beta, \gamma$ and $\delta$.  Substituting these
expressions for each $\kappa_i$, $i=0,1,\ldots,d$, we obtain
an ideal  $P_{d,C}$ of height $3$ in $\Q[\alpha,\beta,\gamma,\delta]$.
The complex projective variety of $P_{d,C}$ is the desired set $\mathcal{T}_C$,
and we can compute  the Chow form (\ref{eq:chowform}) from the
generators of $P_{d,C}$ by a standard elimination process.

\begin{remark}[$d=6$]
The prime ideal $P_6$ is minimally generated by $45$ quartics such as
$$
16 \kappa_0^2 \kappa_6^2+8 \kappa_0 \kappa_1 \kappa_5 \kappa_6-4 \kappa_0 \kappa_3^2 \kappa_6+\kappa_1^2 \kappa_5^2 \,,\,\,
8 \kappa_0^2 \kappa_5 \kappa_6+2 \kappa_0 \kappa_1 \kappa_5^2-4 \kappa_0 \kappa_2 \kappa_3 \kappa_6+\kappa_1^2 \kappa_3 \kappa_6\, ,\,\, \ldots 
$$
The variety of the ideal 
$P_6$ consists of all binary sextics that are squares of binary cubics,
$$ \kappa_0 x_1^6 + \kappa_1 x_0^1 x_1^5 + 
\kappa_2 x_0^2 x_1^4 + \cdots + \kappa_6 x_0^6 \,\, = \,\,
(\nu_0 x_1^3 + \nu_1 x_0 x_1^2 + \nu_2 x_0^2 x_1 + \nu_3 x_0^3)^2 , $$
and it is quick and easy to generate all $45$ generators of $P_6$
in a system like {\tt Macaulay2}. \qed 
\end{remark}

We now discuss the ideal $P_{d,C}$ and the Chow form (\ref{eq:chowform})
for two examples from Section 3.

\begin{example}
The systematic computation of the
tritangent planes to the curve $C$ in Figure~\ref{fig:frank} is done as follows.
First, the trigonometric parametrization $({\rm cos}(\theta), {\rm sin}(2 \theta),
{\rm cos}(3 \theta))$ is made polynomial via (\ref{eq:coordchange}),
and the resulting binary form (\ref{eq:binaryform}) is then found to be
$$
(\alpha+\beta+\delta) x_0^6
+4 \gamma x_0^5 x_1
+(3 \alpha+\beta-15 \delta) x_0^4 x_1^2
+(3 \alpha-\beta+15 \delta) x_0^2 x_1^4
-4 \gamma x_0 x_1^5
+(\alpha-\beta-\delta) x_1^6.
$$
Substituting the coefficients for $\kappa_0,\ldots,\kappa_6$
into the $45$ generators of $P_6$, we arrive at the ideal
\begin{eqnarray*}
P_{6,\gamma} \,\, = &
    \langle \alpha-\delta,\beta,\gamma\rangle \,\cap \,
   \langle \alpha+\delta,\beta,\gamma\rangle \, \cap \, \
    \langle \alpha+\beta-7\delta, \gamma^2+4\beta\delta, \gamma \delta, \delta^2 \rangle
\\ &      \,\cap\,\,
    \langle \alpha-\beta+7\delta,\gamma^2+4\beta\delta, \gamma \delta, \delta^2\rangle
   \, \,\cap\,\,
    \langle \alpha,\beta,\gamma,\delta\rangle^4.
\end{eqnarray*}
The ideal $P_{6,\gamma}$ has two simple roots and
two triple roots in $\C  \PP^3$, and its Chow form equals
$$ {\rm Chow}(\mathcal{T}_C) \,\, = \,\, (z-1)(z+1)(x-1)^3(x+1)^3. $$
This is the degree eight equation which defines the tritangent planes to the curve $C$. \qed
\end{example}

\smallskip

\begin{example}
Let $C$ be Henrion's curve
$\,( {\rm cos}(\theta) + 
2 \,{\rm cos}(2 \theta), \,{\rm sin}(\theta ) + 2 \,{\rm sin}(2 \theta),\, 2 \, {\rm sin}(3 \theta))\,$
in Example \ref{ex:Henrion}.
Here the Chow form of tritangent planes to $C$ is found to factor as follows:
\begin{small}
\begin{eqnarray*}
(z-2) (z+2) (12 x^3-11 x^2 z-6 x^2 y+8 x z y-36 x y^2+z^3+4 z^2 y-7 z y^2-6 y^3-100 x^2+16 x z
\\ -80 x y+32 z^2+36 z y-92 y^2-16 x+292 z-40 y+752) (12 x^3+11 
x^2 z+6 x^2 y+8 x z y-36 x y^2  -z^3 \\-4 z^2 y+7 z y^2+6 y^3-100 x^2-16 x z+80 x y
+32 z^2+36 z y-92 y^2-16 x-292 z+40 y+752). \qed
\end{eqnarray*}
\end{small}
\end{example}

\smallskip

We end with a prominent example of a sextic trigonometric curve.
Freedman asked in \cite{Fre} whether every generic smooth knotted curve $C$
in $\R^3$ must have a tritangent plane. This question was answered negatively
by Ballesteros and Fuster \cite{BF} and Morton \cite{Mor}.
Their counterexamples are trigonometric curves. What follows is
Morton's trefoil example.

\begin{center}
\begin{figure}
\includegraphics[width=9.2cm]{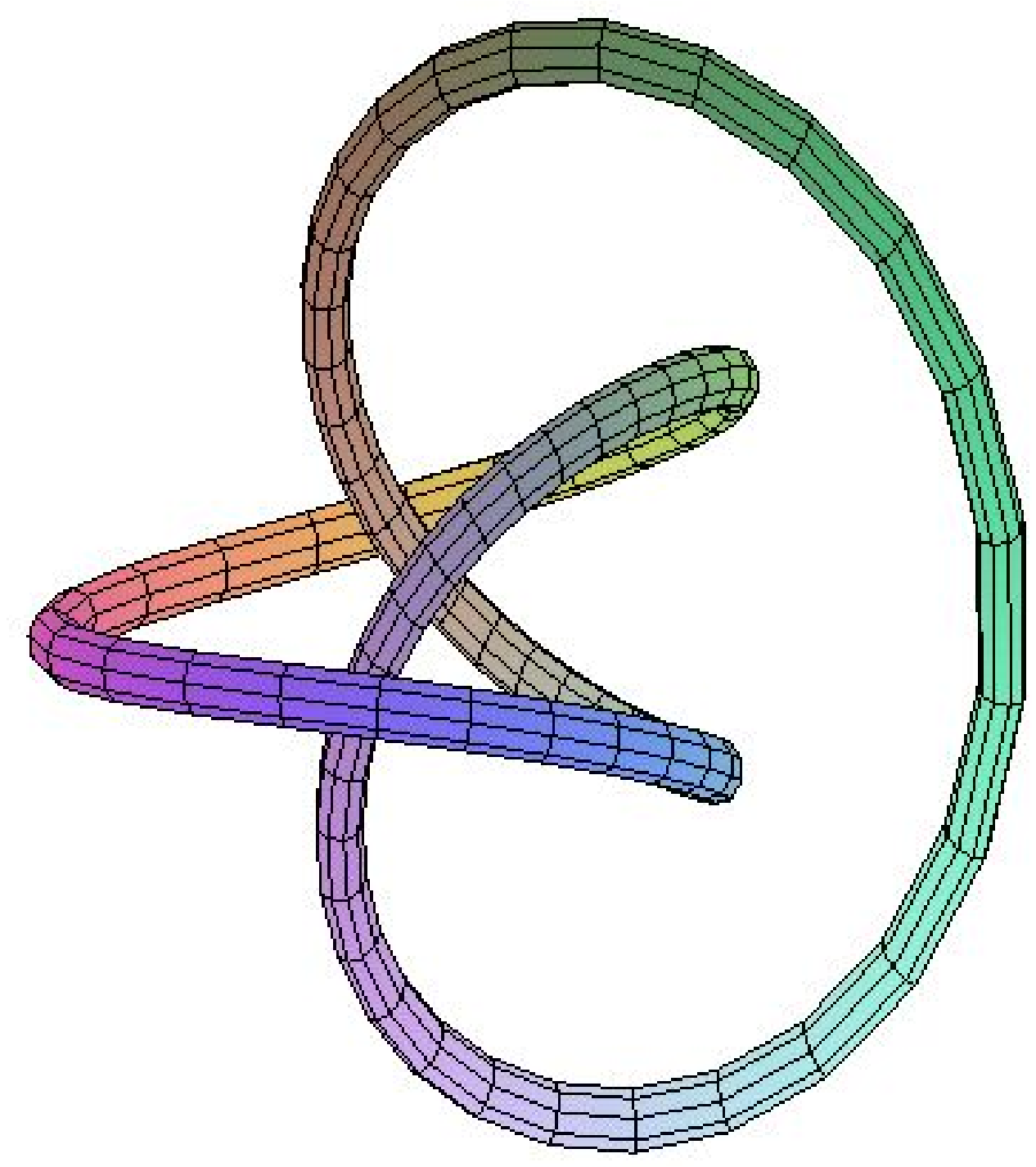} 
\hskip -2.3cm
\includegraphics[width=9.2cm]{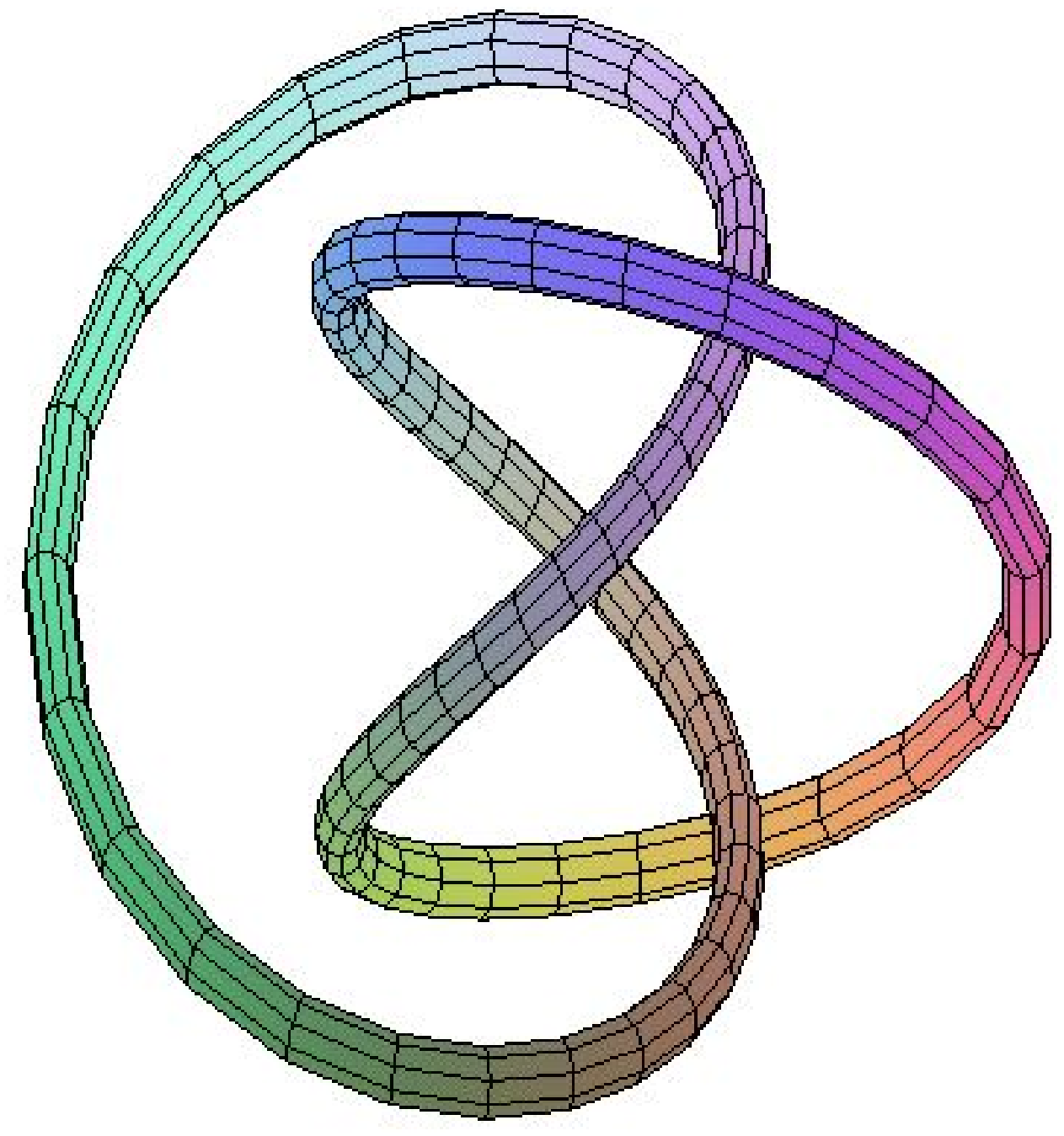}
\vskip -1.4cm
\caption{Morton's trigonometric curve 
without tritangent planes (Example \ref{ex:Morton}).}
\label{fig:morton}
\end{figure}
\end{center}

\begin{example} \label{ex:Morton}
Morton \cite{Mor} constructed the following rational trigonometric curve:
$$ C \,\, : \,\,\theta \,\,\mapsto \, \,
\frac{1}{2- {\rm sin}(2 \theta)}
\bigl( {\rm cos}(3 \theta), {\rm sin}(3 \theta), {\rm cos}(2 \theta) \, \bigr) $$
This  is a trefoil knot in $\R^3$.
The corresponding polynomial parametrization 
$\C \PP^1 \rightarrow \C \PP^3$  is
$$
(x_0:x_1) \,\,\mapsto \,\,
\begin{pmatrix}
2 (x_0^4+2 x_0^2 x_1^2+x_1^4-2 x_0^3 x_1+2 x_0 x_1^3) (x_0^2+x_1^2) \\
(x_0-x_1) (x_0+x_1) (x_1^2+4 x_0 x_1+x_0^2) (x_1^2-4 x_0 x_1+x_0^2) \\
2 x_0 x_1 (x_0^2-3 x_1^2) (3 x_0^2-x_1^2) \\
(2 x_0 x_1+x_0^2-x_1^2) (x_0^2-x_1^2-2 x_0 x_1) (x_0^2+x_1^2)
\end{pmatrix}. $$
We compute the ideal $P_{6,C}$ and its Chow form as described above.
The result is
$$ \begin{matrix}
{\rm Chow}(\mathcal{T}_C) \,\, = & \!\!\!
(x^2+y^2)^2 \cdot \bigl( 13225 x^4+58880 x^3 y+91986 x^2 y^2-638976 x^2 z^2
+13225 y^4 \\ &
+58880 x y^3- 1148160 x y z^2
-638976 y^2 z^2+6230016 z^4+449280 x^2 z \\ &
-449280 y^2 z -409600 x^2-736000 x y-409600 y^2-7987200 z^2+2560000 \bigr).
\end{matrix} $$
The large quartic factor is irreducible over $\mathbb{Q}$ but
it decomposes over $\mathbb{R}$ into linear factors:
$$ \begin{matrix}
(1+0.339305 \,x+0.211829 \,y+1.248999 \,z) \cdot
(1-0.339305 \,x-0.211829 \,y+1.248999 \,z) \cdot \\ 
(1+0.211829 \,x+0.339305 \,y-1.248999 \,z) \cdot
(1-0.211829 \,x-0.339305 \,y-1.248999 \,z) .
\end{matrix} $$
Each of these four tritangent planes touches the curve
$\bar C$ in one real point and in two imaginary points.
This answers Freedman's question:
the real curve $C \subset \R^3$ has no tritangent planes.
Hence the algebraic boundary $\partial_a {\rm conv}(C)$
consists only of the edge surface of $C$. We find that this surface
decomposes (over $\mathbb{Q}$) into
two components of degrees $10$ and $20$.
\qed
\end{example}

\section{Degree formulas for Smooth and Singular Curves}

This section concerns the algebraic geometry of the edge surface.
We present a self-contained derivation
of its degree and the degrees of its curves of singularities.  
Our approach to these  calculations uses
a method that goes back to Cayley and Pl\"ucker. 
In particular, Cayley computed in  \cite{Cay} the
degree of the curve dual to 
the edge surface, i.e. the number of planes through a general point that are tangent to the curve at two points.
Our approach is based on two classical formulas:  
 De Jonqui\`eres' formula (\ref{eq:dejonquieres}) 
and Hurwitz' formula $R=2\delta+2g-2$ for the degree 
$R$ of the ramification 
of a surjective map of degree $\delta$ from a
smooth curve of genus $g$ onto $\C\PP^1$ (see \cite[IV.2.4]{Ha}).  
The degree formulas we find for smooth curves are not new.
In recent years they appeared in the study by  
von zur Gathen \cite{JvzG} of secant spaces to space curves, 
in the classification of projection centers of plane projections by 
T.~Johnsen \cite {J},
and in the study of the focal surface of a congruence of lines by 
 Arrondo {\it et al.} \cite{ABT}. The
correction term  for singular curves 
in Theorem \ref{thm:formula} appears to be new.

The computation interprets the edge surface as the scroll of stationary bisecants to the curve $\bar C$.  
This scroll is a curve of secants, while the variety of all secants  to $\bar C$ form a surface in the 
Grassmannian $\G (2,4)$.  Let $S_{2}=S_{2}\bar C$ be the symmetric square of the curve $\bar C$. If $\bar C$ is singular, we first normalize, 
and take the symmetric square of the normalized curve, thus $S_{2}$ is a smooth surface.
As in Section \ref{trig}, our key object is the secant map 
$S_{2}\to \G (2,4),$
which maps a pair of points on $\bar C$ to the secant  spanned by them.  The image of this map is the secant surface.
We thus  consider the edge surface as a curve of stationary bisecants on $S_{2}$ mapped by the secant map into the Grassmannian.  
We shall determine the class of this curve and of the hyperplane 
section of the secant surface in the N\'eron-Severi group of $S_{2}$.

For a generic smooth curve of genus $g>0$, the N\'eron-Severi group of divisors on $S_{2}$ has rank two.
It is generated by the class of $C_{p}=\{(p,q)|q\in C\}\subset S_{2}$
and half the class of the diagonal, which we denote by $\Delta$. If $\bar C$ has geometric genus $g$, then $\Delta^2=1-g$, while $C_{p}^2=C_{p}\cdot\Delta=1$.   
For rational curves $(g=0)$, we have $S_{2}=\C\PP^2$ and the
N\'eron-Severi group has rank $1$, generated by the class of a line which 
coincides with $C_{p}$ and $\Delta$. The class $H$ on $S_{2}$ of a 
hyperplane section of the secant surface is computed using the fact that the 
lines meeting a fixed line in $\C\PP^3$ form a hyperplane 
section of $\G (2,4)\subset\C\PP^5$.
 The number of secants  through a point on $\bar C$ that intersect a fixed line is $d-1$, so 
$H\cdot C_{p}=d-1$.  The projection of $\bar C$ from a fixed line has ramification at the points where the tangent meets this line.
By Hurwitz' formula, we have  $\,H\cdot 2\Delta=2d+2g-2$, and hence
 $\,H\equiv dC_{p}-\Delta$.

The class $B$ on $S_{2}$ of the curve of stationary bisecants to $\bar C$ is computed similarly: Set $B\equiv aC_{p}+b\Delta$.
By Hurwitz' formula, there are $2(d-2)+2g-2$ stationary bisecant lines through every point, so $B\cdot C_{p}=a+b=2(d+g-3)$. 
Next, a tangent line at a point $p$ is a stationary bisecant line only if $\bar C$ has a plane of third order contact at $p$, i.e. a stall.
 By Remark \ref{stall} their number is $4(d+3g-3)$, so the intersection number of 
 $B$ with the diagonal is $B\cdot 2\Delta=2a+2b(1-g)=4(d+3g-3)$. Hence $B\equiv 2(d+g-1)C_{p}-4\Delta$ when $g>0$. 
 When $g=0$, we have $C_{p}=\Delta$, and $B\equiv 2(d-3)C_{p}$. 

\begin{proposition} \label{degree}  
The degree of the edge surface of a general
smooth space curve of degree $d$ and genus $g$ is  $2(d-3)(d+g-1)$.
\end{proposition}

\begin{proof} 
 The degree of the edge surface is computed by the number of stationary bisecants that intersect a general line, 
so it is $H\cdot B=2(d-1)(d+g-1)-4d+4(1-g)=2(d-3)(d+g-1)$.
 \end{proof}

If $\bar C$ has singularities, then the 
secant map is not defined at the singular points.  
More precisely, it is rational on the symmetric square  $S_{2}$ of the
normalization, but it is not defined on the pairs of points that lie over the singular points on $\bar C$. 
The secant map may, however, be extended by a blowup of $S_{2}$ in the points corresponding to the singularities. 
If the singularities are ordinary nodes or cusps, 
in the sense that no plane has local intersection multiplicity more than $4$,
then the degree of the edge surface may be computed as above.

 \begin{theorem} \label{thm:formula}
The edge surface of a general irreducible space curve of degree $d$, 
geometric genus $g$, with $n$ ordinary nodes and $k$ ordinary cusps,
has degree $2(d{-}3)(d{+}g{-}1)-2n-2k$.
The cone of bisecants through each 
cusp has degree $d{-}2$ and is a component of the surface.
 \end{theorem}

 \begin{proof}
Consider  an ordinary node or an ordinary cusp on $\bar C$, and let
$(p,q)$ (resp. $(p,p)$) be 
the point on $S_{2}$ corresponding to the points on the normalization  
of  $\bar C$ lying over the node (resp. cusp).  
 Let $\tilde S_{2}$ be the blowup $S_{2}$ in $(p,q)$ (resp. $(p,p)$) with exceptional divisor $E$, and let $C_{p}$ and $\Delta$ be the pullback of the respective classes from $S_{2}$.
Let $B$ denote the class on $\tilde S_{2}$ of the total transform of the curve of stationary bisecants on $S_{2}$, and let $H$ denote the class on $\tilde S_{2}$ of a hyperplane section.
Now, the tangent cone at the node (resp. cusp) spans a plane, the unique plane with intersection multiplicity $4$ at the singularity. 
The pencil of lines in this plane through the singular point form 
the image of the exceptional divisor $E$ under the secant map.
 This pencil forms a line in $\G (2,4)$, so $H\cdot E=1$.  
On the other hand, $H\cdot C_{p}=d$ and
  $H\cdot 2\Delta=2d+2g-2$ as before, by
Hurwitz' formula for the normalized curve. Therefore $H\equiv dC_{p}-\Delta-E$.
  Next,  the strict transform of the curve of stationary bisecants  
on the blowup of $S_{2}$ clearly lies in the class of $B-aE$ for some 
positive integer $a$.   The computations of 
 $(B-aE)\cdot C_{p}$ and $(B-aE)\cdot 2\Delta$ are not changed from 
the argument prior to Proposition  \ref{degree}
since we work on the normalization of $\bar C$:
   A secant  through the 
point $p$ belongs to the curve of stationary bisecants on $S_{2}$ only if the tangent at $p$ and the tangent at some other point on the secant  span a plane, 
 while a tangent line at $p$ is a stationary bisecant only if there is third order contact with the branch of $\bar C$ at $p$ with a plane. 

For either singularity, the cone that contains the curve and has its vertex at the singularity has degree $d-2$. The corresponding curve in the blowup of $S_{2}$ 
  lies in the class of $C_{p}-E$.
 The projection from a nodal (resp. cuspidal) tangent 
has degree $d-3$ and ramification (by Hurwitz) of degree $2(d+g-4)$.  This degree counts the number of 
 tangent lines that meet the projection center.  Since the singularities are ordinary, none of these lines are tangent at the singularity.  
The cone curve and the curve of 
stationary bisecants intersect only away from $E$, and 
with intersection number $(B-aE)\cdot (C_{p}-E)=2(d+g-3)-a=2(d+g-4)$.
   Therefore $a=2$, and the degree of the edge surface drops by $2$ 
compared to a smooth curve.

When the singularity is a cusp, the tangent to the normalization 
at the cusp is contracted on $\bar C$, 
so any secant  through the vertex is necessarily a stationary bisecant.  
This curve of secants  form a cone of degree $d-2$ that is a component of 
the edge surface. 

The arguments used in this computation depend only on the local data 
of the singularities, so, as long as the curve $\bar C$ is irreducible, 
the arguments extend to several nodes and cusps.
 \end{proof}

\begin{figure}
\includegraphics[width=9.2cm]{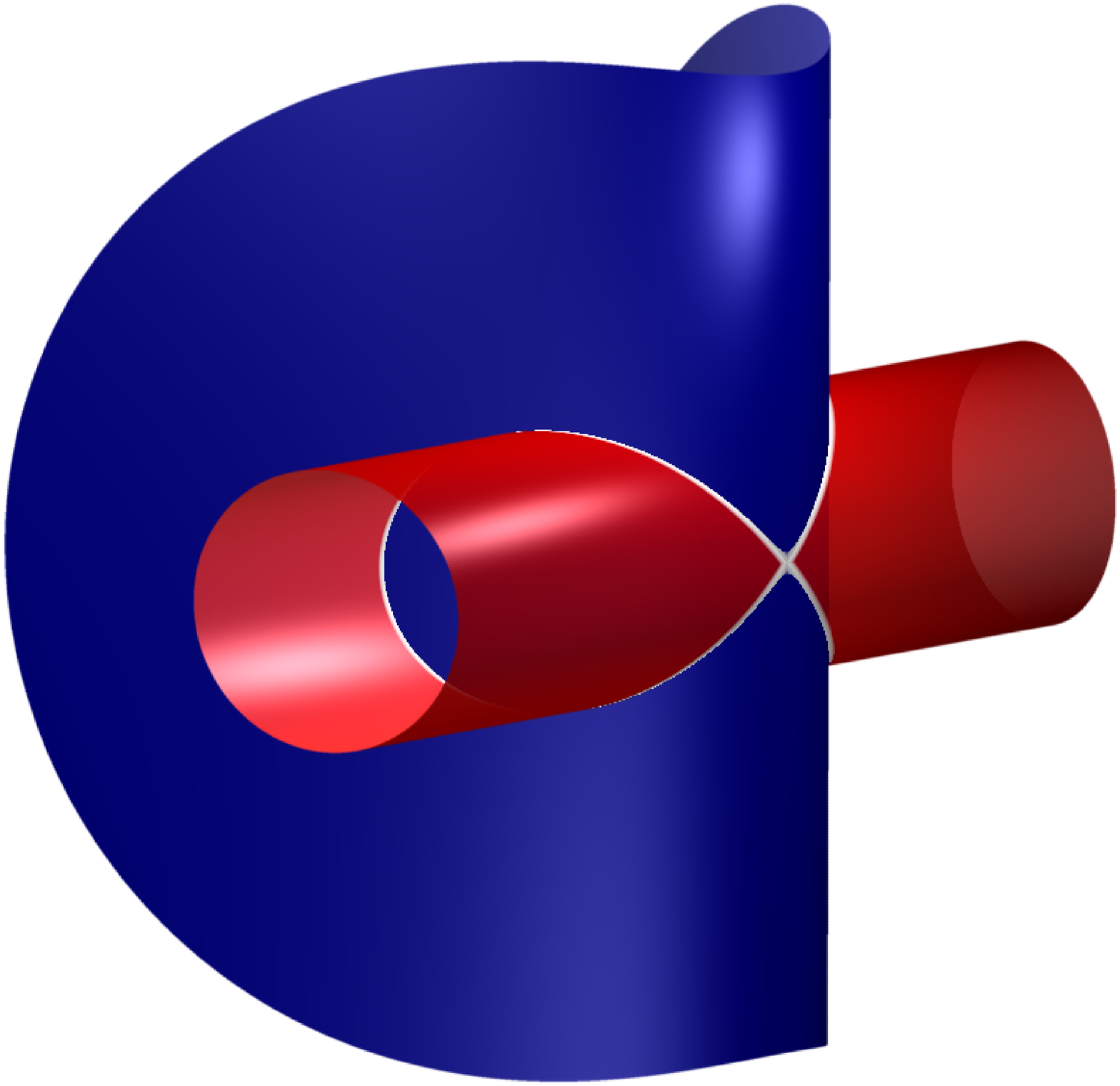} 
\hskip -2.3cm
\includegraphics[width=9.2cm]{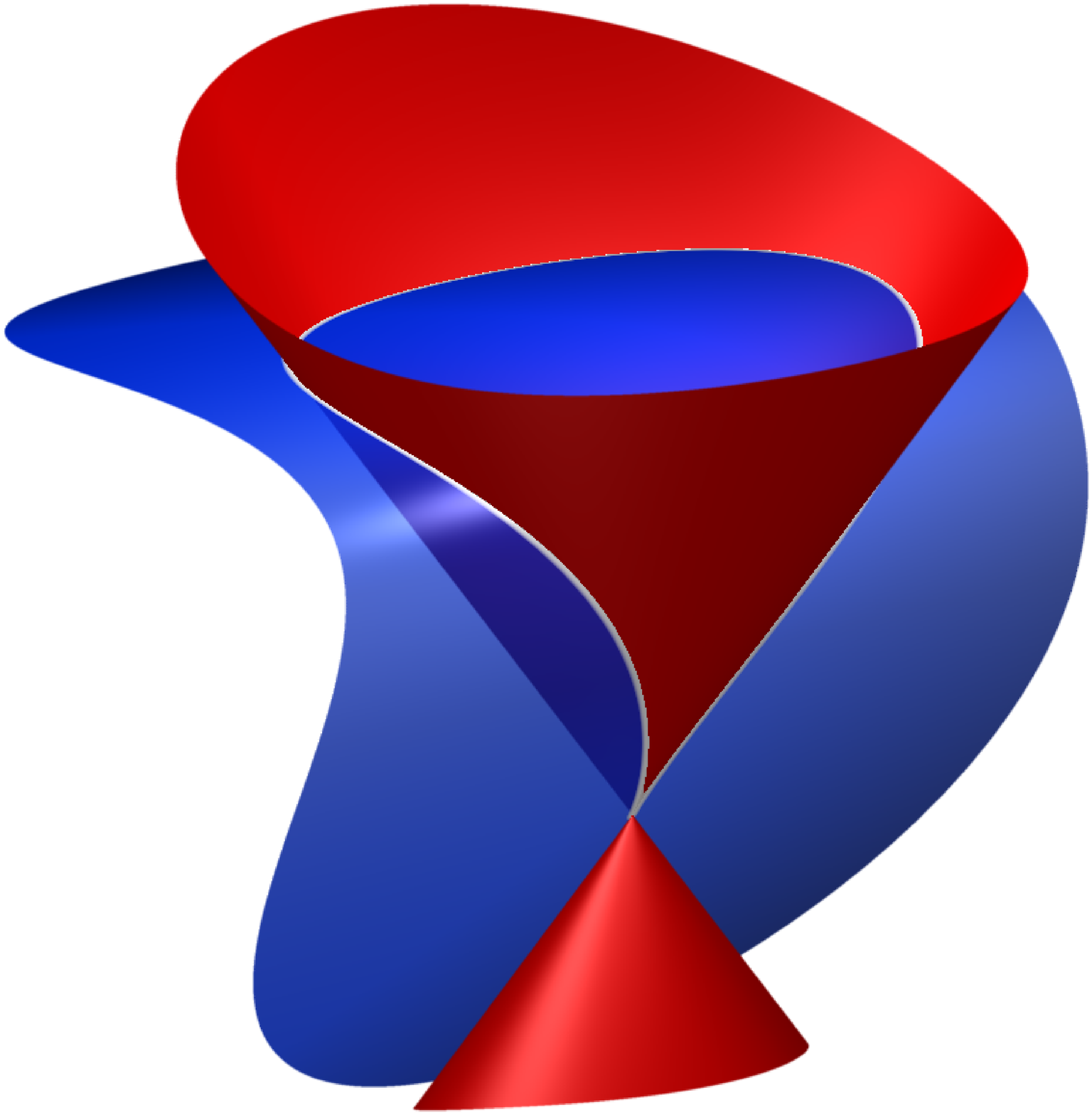}
\vskip -0.4cm
\caption{A nodal and a cuspidal rational quartic curve.}
\label{fig:labs}
\end{figure}

 \begin{example}
Figure \ref{fig:labs} depicts
the edge surface of a rational quartic curve with one 
ordinary singular point.
The edge surface has degree $4$, the value
obtained for $d=4, g=0, n+k=1$ in Theorem \ref{thm:formula}.
It is the union of two quadric cones whose
intersection equals the curve.
If the singularity is an ordinary cusp, then one of the two 
quadrics has its vertex at the cusp. If the singularity is
an ordinary node, then there
are three quadric cones that contain the curve.  
One has its vertex in the node, 
but the edge surface is formed by the other two.
\qed  \end{example}

\begin{remark}
The singularities seen in some of our earlier examples are
not ordinary, and Theorem  \ref{thm:formula} does not apply there.
Henrion's curve in Example \ref{ex:Henrion} has a node in 
the plane at infinity. 
Both branches have intersection multiplicity three with the plane,
so this node is not
ordinary.
     The curve of stationary bisecants has a triple point at 
    the corresponding point of $S_{2}=\C\PP^2$.
Thus the degree of the edge surface is reduced by three from the expected $30$.

The edge surface in Example \ref{ex:runningex}
(shown in Figure~\ref{fig:frank}) has three components, namely two cones 
of degree $2 $ and $3$ respectively, and one component of degree $16$.
 The curve is the complete intersection of the two cones and has two double points.  
 One double point is a node at the vertex of the quadric cone.  The two branches at the point span a plane, but both branches have intersection multiplicity three with this 
 plane so the node is not ordinary.  The other double point lies in the plane at infinity and has two branches with a common tangent.
  The curve of stationary bisecants in $\C\PP^2$
  has three components, two lines corresponding to the two cones 
and one quartic curve.   The secant map has two basepoints on $\C\PP^2$, corresponding to the two singularities.
   The quartic curve has nodes at these basepoints. This explains 
the degree $16$ of the third component.
\qed \end{remark}

These examples illustrate the general fact that the 
dual variety of the edge surface is a curve in the dual $\C \PP^3$.
This is well known, but for lack of reference we include a short proof.
As above, cuspidal points on the curve $\bar C$ play a particular role.  Here we take a cusp to mean any point $P\in\bar C$ such that
 the normalization $\tilde C\to\bar C$ is ramified over $P$, i.e. 
there is some point $Q\in \tilde C$ at which the
normalization is not an immersion. 

\begin{proposition}\label{comp} The variety dual to the edge surface
of any space curve is a curve.  
In particular, each component of the edge surface 
is either a cone or the tangent developable of a curve. 
A cone is a component of the edge surface if and only if it is a cone 
of secants  with vertex at a cusp, or the general ruling intersects the curve 
twice outside the vertex. 
 \end{proposition}

 \begin{proof} 
Let $B_{0}$ be a component of the curve of stationary bisecants  and let $\tilde S_{0}\to S_{0}$ be a minimal desingularization 
of the corresponding component $S_{0}$ of the edge surface.  Then $\tilde S_{0}$ is 
a $\C\PP^1$-bundle over the normalization $\tilde B_{0}$ of  $B_{0}$ with a birational 
morphism onto the scroll $S_{0}$ that contains $\bar C$ in $\C\PP^3$.
Let $L$ be a general line in the ruling of this scroll, and assume first that $L$ does not pass through a cusp on $\bar C$.  
 Then  $L$ is a secant  between distinct points on $\bar C$ whose tangents 
span a plane $P$.   The pullback of this plane on $\tilde S_{0}$ is a curve $H$ that decomposes into $H=A+b\tilde L$ for some 
positive integer $b$ and is singular at the two points of tangency on the pullback $\tilde L$ of the line $L$.  

Now, the curve $H$ is singular at a point of $\tilde L$ only if this is a point of intersection between the component $A$ and $\tilde L$ or $b>1$.
 But $H$ will intersect the general ruling in only one point, so $A$  cannot intersect $\tilde L$ in more than one point.
Therefore $A+b\tilde L$  is singular along the whole ruling, i.e. $b>1$ and the plane $P$ is tangent to the scroll along $L$.
If $L$ is a secant line  through a cusp, the component of the edge surface is a cone of secants  
with vertex at a cusp, and the plane $P$ is tangent to the scroll along $L$.
 Thus the tangent plane 
is constant along each ruling, 
and the dual variety is a curve. 
To complete the proof,  we consider biduality:  Our scroll is the dual variety of a curve, so it is a cone if the curve is planar,
 and it is the tangent developable of the curve of osculating planes if the curve is not planar. 
 
 Since any cone has constant tangent planes along the rulings, a cone is a component of the edge surface for any curve that meet the general ruling in at least two 
 distinct points outside the vertex.  The only other way a cone could be a component of the edge surface is when the vertex is on
  the curve and the general ruling intersect the curve in one more point. To see this, we assume the vertex point is not a cusp on the curve.
 By assumption, each secant line from the vertex point is a 
stationary bisecant line, so there is a plane for each tangent to the curve that contains a tangent line at the vertex.  
  Therefore the projection of the curve from a tangent line at the vertex  is ramified everywhere on the curve.
  This is impossible over $\C$, and the proof is complete.
 \end{proof}

Even when the curve $\bar C$ is generic and smooth, the edge surface is
always highly singular \cite{Sed}. First of all, 
the edge surface has multiplicity $2(d+g-3)$ along $\bar C$ itself.
This number can be derived by applying Hurwitz' formula to the 
projection from a tangent line. In addition, the singular
locus of the edge surface has two further
$1$-dimensional components.  The edge surface is in general the tangent developable of a curve.  
A plane section of the edge surface have cusps at this curve, so the edge surface has a cuspidal edge along this curve. 
Finally the edge surface has in general an additional double curve, where two sheets of the surface intersect transversally.
We compute the degrees of the cuspidal edge and of the double curve.   

\begin{proposition}
Let $\bar C$ be a general smooth curve of degree $d$ and genus $g$.
The edge surface, which is reduced and irreducible,
 has multiplicity $2(d+g-3)$ along $\bar C$ and the 1-dimensional singular locus 
contains in addition a cuspidal edge of degree $6((d+g-3)^2-4g)$ and a double curve of degree
$2d^4+4d^3g+2d^2g^2-18d^3-14dg^2-32d^2g+46d^2+52dg+8g^2-6d+64g-72$.
\end{proposition}
\begin{proof} The dual variety of the edge surface is a curve.  The strict dual to this curve, the curve of osculating planes, is the cuspidal edge of the edge surface.  
Hence both the dual curve and the cuspidal edge are birational to the curve of stationary bisecants. Their common geometric genus is found to be 
$\gamma=2(d{+}g{-}2)(d{+}2g{-}4)+d-7g-4$ by adjunction on the 
symmetric square $S_{2}{\bar C}$.    The degrees of these curves are 
computed by combining de 
Jonqui\`eres' formula with a characterization of the cusps as in 
\cite[Remarks 5.1 and 5.2]{J}.
The degree of the double curve is computed using the double point formula for a plane curve.  The general plane section
  of the edge surface has geometric genus  $\gamma$, it has 
$d$ points of multiplicity $2(d{+}g{-}3)$ and $6((d{+}g{-}3)^2-4g)$ cusps, 
so the formula for the number of double points follows.
  \end{proof}

\bigskip

\noindent
{\bf Acknowledgments.}
We are grateful to Oliver Labs and
Philipp Rostalski for helping us with the diagrams.
Figures \ref{fig:quartic} and
 \ref{fig:labs} were drawn with Labs' software {\tt Surfex},
which is freely available at {\tt www.surfex.algebraicsurface.net}.
We thank Melody Chan, Joao Gouveia, Trygve Johnsen, Clint McCrory, Ragni Piene, Frank Sottile and Cynthia Vinzant for helpful discussions.
Bernd Sturmfels was supported in part by NSF grant DMS-0757207.

\end{document}